\documentclass[11pt]{amsart}
\usepackage[all]{xy}
\usepackage{fullpage}
\usepackage{latexsym}
\usepackage{geometry}
\usepackage{amsmath}
\usepackage{amsfonts}
\usepackage{amssymb}
\usepackage{amsthm}
\usepackage{eucal}
\usepackage{enumerate,yfonts}
\usepackage{mathrsfs}
\usepackage{graphicx}
\usepackage{graphics}
\usepackage{epstopdf}
\usepackage{amscd}
\usepackage{bbm}
\usepackage{hyperref}
\usepackage{url}
\usepackage{color}

\usepackage{bbm}
\usepackage{cancel}
\usepackage{enumerate}
\usepackage{amsmath,amsthm}
\usepackage{amssymb}
\usepackage{epsfig}
\usepackage{pstricks}
\usepackage{xy}

\DeclareSymbolFont{AMSb}{U}{msb}{m}{n}
\DeclareMathSymbol{\N}{\mathbin}{AMSb}{"4E}
\DeclareMathSymbol{\Z}{\mathbin}{AMSb}{"5A}
\DeclareMathSymbol{\Q}{\mathbin}{AMSb}{"51}
\DeclareMathSymbol{\I}{\mathbin}{AMSb}{"49}
\DeclareMathSymbol{\F}{\mathbin}{AMSb}{"46}

\newcommand{\ZC}{\hat{\mathfrak{Z}}_{crit}}
\newcommand{\ZZ}{\mathfrak{Z}_{crit}}
\newcommand{\ZL}{\mathfrak{Z}_{crit,\lambda}}

\newcommand{\ZU}{\mathfrak{Z}_{crit}^{unr}}
\newcommand{\VACO}{\mathbbm{V}^{0}}
\newcommand{\VAC}{\mathbbm{V}^{\lambda}}

\newcommand{\zl}{\mathfrak{z}_{crit,\lambda}}
\newcommand{\zz}{\mathfrak{z}_{crit}}

\newcommand{\OO}{\Omega(\ZZ)}

\newcommand{\AREN}{\mathcal{A}^{ren,\tau}_{crit}}
\newcommand{\ARENL}{\mathcal{A}^{ren,\tau}_{crit,\lambda}}
\newcommand{\AFT}{\mathcal{A}^{\flat,\tau}_{crit}}
\newcommand{\AFTL}{\mathcal{A}^{\flat,\tau}_{crit,\lambda}}
\newcommand{\AF}{\mathcal{A}^{\flat}_{crit}}
\newcommand{\AFL}{\mathcal{A}^{\flat}_{crit,\lambda}}

\newcommand{\hgc}{\widehat{\check{\mathfrak{g}}}_{crit}}
\newcommand{\gc}{\hat{\mathfrak{g}}_{crit}}
\newcommand{\DD}{\mathcal{D}_{crit}^{0}}
\newcommand{\AC}{\mathcal{A}_{crit}}
\newcommand{\AH}{\mathcal{A}_{\hbar}}
\newcommand{\ACL}{\mathcal{A}_{crit,\lambda}}

\newcommand{\PDDm}{(\Psi\boxtimes \Psi)(\DD)\text{-mod}}

\newcommand{\OP}{\text{Op}_{\check{\g},X}}
\newcommand{\OPL}{\text{Op}_{\check{\g},X}^{\lambda}}
\newcommand{\OPPD}{\text{Op}_{\check{\g}}(D^{\times}_{x})}
\newcommand{\OPD}{\text{Op}_{{\check{\g}}}(D_x)}
\newcommand{\OPDL}{\text{Op}_{\check{\g}}^{\lambda,reg}}
\newcommand{\OPDM}{\text{Op}_{\check{\g}}^{\mu,reg}}
\newcommand{\OPDU}{\text{Op}_{\check{\g}}^{unr}}

\newtheorem{thm}{Theorem}[subsection]

\newtheorem{lemma}[thm]{Lemma}

\newtheorem{prop}[thm]{Proposition}

\theoremstyle{definition}

\newtheorem{remark}[thm]{Remark}

\numberwithin{equation}{subsection}

\newcommand{\DG}{\mathcal{D}_{crit}}
\newcommand{\DGM}{\DG\text{-mod}^{G[[t]]}}
\newcommand{\GR}{\text{Gr}_{G}}
\newcommand{\GRK}{G(\!(t)\!)/K}

\newcommand{\g}{\mathfrak{g}}

\newcommand{\DX}{\mathcal{D}_{X}}

\newcommand{\DM}{D_{crit}\text{-mod}(\GR)}
\newcommand{\DMH}{D_{crit}^{Hecke}\text{-mod}(\GR)}
\newcommand{\gmr}{\gc\text{-mod}_{reg}}
\newcommand{\gm}{\gc\text{-mod}}
\newcommand{\il}{i^!_{\lambda}}

\newcommand{\ild}{i_{!,\lambda}}
\newcommand{\M}{\mathcal{M}}
\newcommand{\NN}{\mathcal{N}}

\DeclareMathOperator{\Hom}{Hom}

\newcommand{\B}{\mathcal{B}}
\newcommand{\A}{\mathcal{A}}

\newcommand{\mF}{\mathcal{F}}

\newcommand{\mM}{\mathcal{M}}
\newcommand{\mO}{\mathcal{O}}
\newcommand{\mL}{\mathcal{L}}

\newcommand{\mN}{\mathcal{N}}

\newcommand{\is}{\simeq}
\newcommand{\ra}{\rightarrow}

\newcommand{\fg}{\mathfrak{g}}

\newcommand{\sX}{\mathscr{X}}
\newcommand{\sY}{\mathscr{Y}}

\newcommand{\quash}[1]{} 

\newcommand{\bC}{{\mathbb C}}

\newcommand{\UG}{U'_{crit}}

\title[Twisted global section functor for $D$-modules on affine Grassmannian]{Twisted global section functor for $D$-modules on affine Grassmannian}
        \author{Tsao-Hsien Chen, Giorgia Fortuna}
        \date{\today}
        \address{}
        \email{tsaohsien@gmail.com}
        \address{}
         \email{fortuna@math.mit.edu}
\begin{document}
\maketitle
\begin{abstract}
For each integral dominant weight $\lambda$, we construct a twisted global section functor $\Gamma^{\lambda}$ from 
the category of critical 
twisted 
$D$-modules on affine Grassmannian to the category of $\lambda$-regular modules of 
affine Lie algebra at critical level. 
We proved that $\Gamma^{\lambda}$ is exact and faithful. This generalized the 
work of Frenkel and Gaitsgory \cite{FG} in the case when $\lambda=0$.
\end{abstract}
\setcounter{tocdepth}{1}
 \tableofcontents

\section*{Introduction}
\subsection{} Let $\g$ be a simple Lie algebra over the complex numbers, and $G$ be the corresponding algebraic group of adjoint type.  Let $\kappa$ be an invariant non-degenerate bilinear form $$\kappa:\g\otimes \g\rightarrow \mathbbm{C}.$$
Let $\hat{\g}_{\kappa}$ be the affine Kac-Moody algebra given as the central extension of the loop algebra $\g(\!(t)\!)$
\begin{equation}\label{gc}
0\rightarrow \mathbbm{C}\mathbbm{1}\rightarrow \hat{\g}_{\kappa}\rightarrow \g(\!(t)\!)\rightarrow 0,
\end{equation}
with bracket given by
\[
 [af(t),bg(t)]=[a,b]f(t)g(t)+\kappa(a,b)\text{Res}(fdg)\cdot\mathbbm{1},
\]
\noindent where $a$ and $b$ are elements in $\g$, and $\mathbbm{1}$ is the central element.\\
 Denote by $\gc$ the
affine  Kac-Moody algebra corresponding to $\kappa=-1/2\kappa_{kill}=:\kappa_{crit}$, where $\kappa_{kill}$ denotes the Killing form, and let $\ZC$ be the center of 
appropriately completed twisted enveloping algebra $\UG$ of $\gc$. We are mostly interested in the category $\gm$ of continuous $\UG$-modules. These are the same as discrete $\gc$-modules on which the central element $\mathbbm{1}$ acts as the identity. 
\subsection{} Let $\GR=G(\!(t)\!)/G[[t]]$ be the affine Grassmannian of $G$. Denote by $\DM$ the category of critical-twisted $D$-modules on $\GR$ as introduced in \cite{BD2}. We have the functor of global sections
$$\Gamma: \DM\rightarrow \gc\text{-mod},\,\,\,\,\;\mathcal{F}\mapsto \Gamma(\GR,\mathcal{F}). $$
In \cite{FG}, it is shown that the above functor $\Gamma$ is exact and faithful. This functor plays an important role in the study of the \emph{Geometric Langlands}. More precisely, 
Let $V^{\lambda}$ be the irreducible $\fg$-module with highest weight $\lambda$
and set $\mathbb V^{\lambda}=\text{Ind}_{\fg[[t]]}^{\gc}(V^{\mu})$. Denote by $\zl$ the algebra $\zl:=\text{End}(\mathbb V^{\mu})$. It can be shown that $\zl$ is in fact commutative, and the map $\ZC\rightarrow \zl$ is a surjection
\[
\ZC\twoheadrightarrow \zl.
\]
Denote by $\gm_{reg,\lambda}$ the subcategory of $\gm$ consisting of modules such that the action of the center 
$\ZC$ factors through $\zl$. For $\lambda=0$ we will simply denote $\mathfrak{z}_{crit,0}$ and $\gm_{reg,0}$ by $\zz$ and $\gmr$ respectively.\\
\indent It can be shown that the functor $\Gamma: \DM\rightarrow \gm$, in fact, lands in the subcategory $\gmr$. Moreover, as it is shown in \cite{FG5}, the exactness and faithfulness of the above functor, allows to show the equivalence of categories 
\begin{equation}\label{Iwa}
\DMH^{I_{0}}\xrightarrow{\sim}\gmr^{I_{0}}.
\end{equation}
In the above expression, $\DMH$ denotes the \emph{Hecke category} 
\[
\DMH:=\DM\underset{\text{pt}/\check{G}}{\times} \text{Spec}(\zz),
\]
and $\gmr^{I_{0}}$ (resp. $\DMH^{I_0}$) denotes the subcategory of $\gmr$ (resp. of $\DMH$) consisting of modules which are $I_0$-integrable (resp. $I_0$-equivariant ), where $I_{0}$ denotes the unipotent radical of the Iwahori subgroup $I\subset G[[t]]$. 

 \subsection{}
 The main purpose of this paper is to construct a different functor $\Gamma^{\lambda}$
 \[
 \Gamma^{\lambda}: \DM\rightarrow \gm_{reg,\tau(\lambda)},
 \]
where $\tau$ is the involution of the Dynkin diagram that sends a weight $\lambda$ to $-w_0(\lambda)$, for $w_0$ longest element in the Weyl group. We will show that $\Gamma^{\lambda}$, for $\lambda$ dominant, is exact and faithful. Following \cite{FG5}, this will be the point of departure for a following-up paper where an equivalence similar to \ref{Iwa} will be shown. More precisely, the functor $\Gamma^{\lambda}$ can be used to construct a different functor $\Gamma^{\lambda,Hecke}$ yielding an equivalence
\[
\DMH^{I_{0}}\xrightarrow{\sim}\gm_{reg,\tau(\lambda)}^{I_{0}}.
\] 

\subsection{} In order to explain how the functor $\Gamma^{\lambda}$ arises, it is important to recall
the construction of the functor of global sections given in \cite{FG}. For this, recall the \emph{chiral algebra $\DG$ of critically twisted chiral differential operators on the loop group} $G(\!(t)\!)$, introduced in \cite{AG}.  As it is shown in {\it loc. cit.}, it admits two embeddings
\[
\AC\xrightarrow{l}\DG\xleftarrow{r}\AC,
\]
where $\AC$ is the chiral algebra 
attached to the Lie$^*$-algebra $L=\g\otimes \DX\oplus \Omega_{X}$ at the critical level as explained in \cite{BD}.
If we restrict these two embeddings to the center $\ZZ$ of $\AC$, as it is explained in \cite{FG} Theorem 5.4, we have
\[
 l(\ZZ)=l(\AC)\cap r(\AC)=r(\ZZ).
\]
Moreover the two compositions
\begin{equation}\label{centre}
 \ZZ\hookrightarrow \AC\xrightarrow{l}\DG\xleftarrow{r}\AC\hookleftarrow \ZZ
\end{equation}
are intertwined by the automorphism $\tau: \ZZ\rightarrow \ZZ$ mentioned before.\\
One would expect to have some sort of functor between $D$-modules on the affine Grassmannian and $D$-modules on the loop group $G(\!(t)\!)$ that are $G[[t]]$-equivariant. However, the difficulties in defined the category of $D$-modules on the loop group, make the existence of such functor vague. However, as it is explained in \cite{AG},  it is natural to relate $\DM$ with the category $\DGM$ of $G[[t]]$-equivariant $\DG$-modules supported at $x\in X$. In fact the following is true.
\begin{thm}\label{t1}
 There exist a canonical equivalence of categories
\[
 \DM\simeq \DGM.
\]
\end{thm}
\noindent 

Under the above equivalence, the functor $\Gamma$ is given by the composition
\[
 \DM\simeq \DGM\xrightarrow{For} (\gc\times \gc)\text{-mod}^{G[[t]]}\xrightarrow{\Hom(\mathbbm{V}^{0},\,\,)}\gm,
\]
where the forgetful functor to $(\gc\times \gc)\text{-mod}^{G[[t]]}$ is given by the embeddings $l$ and $r$ and by the equivalence between $\AC$-modules supported at $x$ and $\gm$.
In other words, given a $D_{crit}(\GR)$-module $\mathcal{F}$, if we denote by $M_{\mathcal{F}}$ the corresponding $\DG$-module, then 
$\Gamma(\GR,\mathcal{F})$ is given by $(M_{\mathcal{F}})^{\g[[t]]}=\Hom(\mathbbm{V}^{0},M_{\mathcal{F}})$.

\subsection{Statement of the Main Theorem}
The above construction suggests a way of defining a different functor $\Gamma^{\lambda}$ from $\DM$ to the category $\gm_{reg,\tau(\lambda)}$ introduced earlier. 
In fact we can define $\Gamma^{\lambda}$ as the composition
\begin{equation}\label{eq1}
  \DM\simeq \DGM\xrightarrow{For} (\gc\times \gc)\text{-mod}^{G[[t]]}\xrightarrow{\Hom(\VAC,\,\,)}\gm_{reg,\tau(\lambda)}.
\end{equation}

The main theorem of this paper is that, for dominant weights $\lambda$'s, the functor $\Gamma^{\lambda}$ remains exact and faithful.
\begin{thm}\label{th1}
 For any dominant weight $\lambda$, the functor $$\Gamma^{\lambda}:D_{crit}(\GR)\text{-mod}\rightarrow \gc\text{-mod}_{reg,\tau(\lambda)}$$ is exact and faithful.
\end{thm}
The proof of Theorem \ref{th1} follows the line of \cite{FG} and can be divieded into two parts:
\begin{itemize}
\item Showing that for $\gc$-modules $M_{\mF}$ correspnding to $\mF\in\DM$ under Theorem \ref{t1}, 
the functor of taking maximal submodule of $M_{\mF}$ which is supported on $\text{Spec}(\zl)$ is exact.
\item The functor $\Hom(\VAC,\,\,)$ from $\gm_{reg,\lambda}$ to the category of vector spaces $\text{Vect}$ it exact. 
\end{itemize}

In fact, part  (2) follows from a result of \cite{FG} that says that $\VAC$ is projective in $\gm_{reg,\lambda}$.
Therefore, our job in the present note is to prove the claim in part (1).
This will be done using the construction of "modification at a point"
(Proposition \ref{p3}) and a chiral-version of Kashiwara lemma (Proposition \ref{kash 1}).
In the case when $\lambda=0$, which is done in \cite{FG}, we only need Kashiwara lemma to
complete the proof of part (1) and the main reason is that the chiral $\ZZ$-moodules
appearing in this case are central (see \ref{chiral} for the definition). It is no longer 
true for general $\lambda$ and the construction of ``modification at a point'' allows us
to deal with non-central chiral $\ZZ$-modules.

\subsection{}
The paper is constructed as follows. In Section \ref{2} we recall the main results about the center $\ZZ$, the space of opers on a curve $X$ and we 
introduce the construction of "modification at a point " for general chiral algebras. In section \ref{3}
we study Lie$^*$-algebroids and chiral algebroids arising from $\ZZ$. In section \ref{4} we reduce 
the exactness of the functor $\Gamma^{\lambda}$ to a chiral-version of Kashiwara lemma.
In Section \ref{5} we prove the Kashiwara lemma.
In section \ref{6} we prove the faithfulness of $\Gamma^{\lambda}$. 
\subsection{Conventions}

\noindent\label{chiral} Our basic tool in this paper is the theory of chiral algebras. We will assume the reader is familiar with
the foundational work \cite{BD} on this subject. However we will briefly recall some basic definitions and notations.\\
Throughout this paper $\Delta:X\hookrightarrow X\times X$ will denote the diagonal embedding and
 $j:U\rightarrow X\times X$ its complement, where $U=(X\times X)-\Delta(X)$.\\  For any two sheaves $\M$ and $\mathcal{N}$ 
denote by $\M\boxtimes \mathcal{N}$ the external tensor product 
$\pi_1^*\M\underset{\mathcal{O}_{X\times X}}{\otimes}\pi_2^*\mathcal{N}$, where $\pi_1$ and $\pi_2$ are the 
two projections from $X\times X$ to $X$.  For a right $\DX$-module $\M$ define the extension $\Delta_!(\M)$ as
\[
\Delta_!(\M)=j_*j^*(\Omega_X\boxtimes \M)/\Omega_X\boxtimes \M.
\]
Sections of $\Delta_!(\M)$ can be thought as distributions on $X\times X$ with support on the
 diagonal and with values on $\M$. If $\M$ and $\mathcal{N}$ are two right $\DX$-modules, we will denote by 
$\M\overset{!}{\otimes}\mathcal{N}$ the right $\DX$-module $\M\otimes \mathcal{N}\otimes \Omega_X^*$.\\
\newcommand{\R}{\mathcal{R}}
\indent A \emph{chiral algebra} over $X$ is a right $\DX$-module $\A$ endowed with a \emph{chiral bracket}, i.e. with a map of 
$\mathcal{D}_{X^2}$-modules 
\[
 \mu:j_*j^*(\A\boxtimes \A)\rightarrow \Delta_!(\A)
\]
which is antisymmetric and satisfies the Jacobi identity.\\
We will denote by $[\,,\,]_{\A}$ the restriction of $\mu$ to $\A\boxtimes \A\hookrightarrow j_*j^*(\A\boxtimes \A)$, and we will refer to it as the \emph{induced Lie$^*$-bracket}.\\ 
By a \emph{commutative chiral algebra} we mean a chiral algebra $\R$ such that $[\,,\,]_{\R}$ vanishes. 
In other words it is a chiral algebra such that the chiral bracket $\mu$ factors as
\[
 j_*j^*(\R\boxtimes \R)\rightarrow \Delta_!(\R\overset{!}{\otimes}\R)\rightarrow\Delta_!(\R).
\]
\noindent Equivalently, $\R$ can be described as a right $\DX$-module with a commutative product on the corresponding 
left $\DX$-module $\R^l:=\R\otimes \Omega_X^*$.\\
\indent A chiral $\A$-module is a right-$\DX$-module $\M$ endowed with a map 
\[
 \mu_{\A,\M}:j_*j^*(\A\boxtimes \M)\rightarrow \Delta_!(\M),
\]
satisfying certain properties. We call the restriction of $\mu_{\A,\M}$ to $\A\boxtimes \M\hookrightarrow j_*j^*(\A\boxtimes \R)$ the \emph{induced Lie$^*$-action}.

For a commutative chiral algebra $\R$, and a chiral $\R$-module $\M$, we say that $\M$ is \emph{central} if the induced Lie$^*$-action of $\mu_{\R,\M}$ vanishes. Equivalently, $\M$ is a chiral $\R$-module such that the chiral action $\mu_{\R,\M}$ factors as
\[
 j_*j^*(\R\boxtimes \M)\rightarrow \Delta_!(\R\overset{!}{\otimes}\M)\rightarrow\Delta_!(\M).
\]

\subsection{Acknowledgments}
The authors would like to thank 
Dennis Gaitsgory and Roman Bezrukavnikov for suggesting the problem, for 
useful discussion during the development of this paper.  
The first author would like to thank the IAS for excellent
working condition.
The work of T-H.C. is supported by NSF under the agreement No.DMS-1128155.

\section{The center at the critical level and the space of opers.}\label{2}
Recall that, for the critical level $\kappa=\kappa_{crit}$, we denote by $\AC$ the chiral
attached to the Lie$^*$-algebra $L=\g\otimes \DX\oplus \Omega_{X}$ as explained in \cite{BD}, and by $\ZZ$ its center. 
Recall that $\ZZ$ is a commutative chiral algebra whose fiber $\zz:=(\ZZ)_{x}$ at any $x\in X$ is isomorphic to
\[
 \zz\simeq \text{End}(\VACO),
\]
where $\VACO$ denotes the vacuum module for $\gc$ given as
\[
 \VACO:=\text{Ind}_{\g[[t]]\oplus \mathbbm{C}}^{\gc}\mathbbm{C}.
\]
Let $\ZC$ be the topological associative algebra attached to $\ZZ$, introduced in \cite{BD}, 3.6.18.\\
One can show that $\ZC$ is in fact isomorphic to the center of 
 the appropriately completed twisted  enveloping algebra $\UG$ of $\gc$, where  
$\UG$ denotes the quotient $U(\hat{\g}_{crit})/(\mathbbm{1}- 1)$, (here $1$ denotes the identity element in $U(\gc)$).\\
Denote by $\check{\g}$ the Langlands dual Lie algebra to $\g$ and let $\OP$ be the $\DX$-scheme of $\check{\g}$-opers on $X$ introduced in \cite{BD2}. 
For every point
$x\in X$, and coordinate $t$ around $x$, denote by $\OPPD$ the ind-scheme of opers on the punctured 
disc $D^\times_{x}=\text{Spec}(\mathbbm{C}(\!(t)\!))$, and by $\OPD$ the space of regular opers. 
Explicitly, an oper $\nabla\in \OPPD$ is the equivalence class, under the Gauge action of $\check{N}(\!(t)\!)$, of elements of the form
\begin{equation}\label{eq3}
 \nabla=\nabla_0+p_{-1}dt+v(t)dt,
\end{equation}
where $v(t)\in \check{\mathfrak{b}}(\!(t)\!)$ and $p_{-1}$ denotes the element $p_{-1}=\sum_{i\in I}f_{i}$, for $I$ index set of simple roots. The condition that the 
oper is regular, i.e. it belongs to $\OPD$, is given by $v(t)\in \check{\mathfrak{b}}[[t]]$.\\
Recall now the following Theorem, established in \cite{FF}.
\begin{thm}\label{th3a}
There exist a canonical isomorphism of $\DX$-algebras
\[
 \ZZ\simeq \text{Fun}(\OP),
\]
in particular we have an isomorhism of commutative algebras $\zz\simeq \text{Fun}(\OPD)$ and of
 commutative topological algebras $\ZC\simeq\text{Fun}(\OPPD).$
\end{thm}
\subsection{Poisson structure on the space of opers}
Recall now the Poisson structure on $\OPPD$. 
Consider the space $Conn_{\check G}(D^\times_x)$ of all connections on the trivial $\check G$-bundle on 
$D_x^\times$, i.e. operators of the form
\[
 \nabla_0+\phi(t)dt, \text{where } \phi(t) \in \check{\g}\otimes\Omega_{D_x^*}.
\]
If we denote by $\hgc^*$ the topological dual of $\hgc$, then we can identify $Conn_{\check G}(D_x^\times)$ with the hyperplane in
$\hgc^*$ consisting of all functions $h:\hgc\ra\bC$ such that the restriction of $h$ to
the center of $\hgc$ is the identity.
Under this identification, the coadjoint
action of $\check G(\!(t)\!)$ on $\hgc^*$ corresponds to the gauge action of $\check G(\!(t)\!)$ on $Conn_{\check G}(D_x^\times)$. The space $\hgc^*$ carries a natural 
Poisson structure that induces a Poisson structure on $Conn_{\check G}(D_x^*)$.\\
Consider now the action of $\check{N}(\!(t)\!)$ on $\hgc^*$. The map $\mu:\hgc^*\rightarrow (\check{\mathfrak{n}}(\!(t)\!))^*\simeq \check{\g}/\check{\mathfrak{b}}\otimes \Omega_{D_x^*}$ 
is a moment map for this action and in particular on $Conn_{\check{G}}(D_x^\times)$. 
Moreover, from \cite[\S 3.7.14]{BD2}, we have an identification
\[
 \OPPD\simeq(\mu^{-1}(l)\!)/\check N(\!(t)\!).
\]
for any non-degenerate characters $l$ of $\check{\mathfrak{n}}(\!(t)\!)$.
In other words we can construct $\OPPD$ as the Hamiltonian reduction of $Conn_{\check G}(D_x^\times)$ along $\mu$. 
In particular $\OPPD$ acquires a Poisson structure 
$$\{\,,\,\}:Fun(\OPPD)\otimes Fun(\OPPD)\rightarrow Fun(\OPPD).$$
The Poisson structure on $\OPPD$ gives $\Omega(\OPPD)$ a structure of Lie algebroid aver it. We will denote by $\omega$ the resulting 
anchor map $\omega:\Omega(\OPPD)\rightarrow T(\OPPD)$, where $T(\OPPD)$ is the tangent sheaf to $\OPPD$.\\

\noindent Denote by $I_{0}$ the ideal corresponding to $\OPD\subset \OPPD$. Recall that, in \cite{FG3} it is shown that $I_0$ is co-isotropic, i.e. 
that $\{I_0,I_0\}\subset I_0$. In particular $I_0/I_0^2$  is a Lie algebroid over $\OPD$ and we have the following commutative diagram:
\[
 \xymatrix{0\ar[r]& I_0/I_0^2\ar[r]\ar[d]&\Omega(\OPPD)|_{\OPD}\ar[r]\ar[d]^{\omega}&\Omega(\OPD)\ar[r]\ar[d]&0,\\
0\ar[r]&T(\OPD)\ar[r]&T(\OPPD)|_{\OPD}\ar[r]&N_{\OPD/\OPPD}\ar[r]&0}
\]
\noindent where $N_{\OPD/\OPPD}$ denotes the normal bundle of $\OPD\subset\OPPD$ (in particular
 $N_{\OPD/\OPPD}=\left( I_0/I_0^2\right)$).
\subsection{}\label{alg}
Recall now the sub-scheme $\OPDL\subset\OPPD$ of $\lambda$-regular opers, for $\lambda$ dominant coweight of $\g$. This sub-scheme consists of equivalence classes of connections of the form
\begin{equation}\label{eq4}
\nabla=\nabla_0+\left(\sum_{i}t^{<\alpha_i,\lambda>}\cdot f_i\right)dt+v(t)dt,\,\,\,\,\,\mbox{ for $v(t)\in \check{\mathfrak{b}}[[t]]$ }
\end{equation}
under the action of $\check{N}[[t]]$.   In particular $\OPD=\text{Op}_{\g}^{0,reg}.$\\
Note that the space $\OPPD$ of opers on the punctures disc, can be constructed as the Hamiltonian reduction along the map $\mu: Conn_{\check{G}}(D_x^*)\rightarrow \check{\mathfrak{n}}(\!(t)\!)^*\simeq \check{\g}/\check{\mathfrak{b}}\otimes \Omega_{D_x^*}$ by choosing the fiber of the character  $\left(\sum_{i}t^{<\alpha_i,\lambda>}\cdot f_i\right)dt\in \check{\mathfrak{n}}(\!(t)\!)^*$,
\[
\OPPD\simeq(\mu^{-1} \left(\sum_{i}t^{<\alpha_i,\lambda>}\cdot f_i\right)dt )/\check{N}(\!(t)\!).
\]
\noindent Denote by $I_{\lambda}$ the ideal corresponding to $\OPDL\subset \OPPD$. We have the following lemma:
\begin{lemma}\label{l2}
The sub-scheme $\OPDL$ is co-isotropic, i.e. $\{I_{\lambda},I_{\lambda}\}\subset I_{\lambda}$.
\end{lemma}
\begin{proof}
We follow the argument in \cite[Lemma 4.4.1]{FG3}. Let $Conn^{\text{reg}}\subset Conn_{\check G}(D_x^\times)$
be the subscheme of regular connections, i.e., connections of the form $\nabla=d+f(t)dt$ 
where $f(t)\in\check\fg[[t]]$. It is co-isotropic since the ideal sheaf of $Conn^{\text{reg}}$
is generated by $\check\fg[[t]]$ which is a subalgebra.
By (\ref{eq4}), we have 
\[\OPDL\simeq(\mu^{-1} \left(\sum_{i}t^{<\alpha_i,\lambda>}\cdot f_i\right)dt\cap Conn^{\text reg})/\check{N}(\!(t)\!)\]
which implies the claim of the Lemma.
\end{proof}
From the above lemma it follows that $I_{\lambda}/I_{\lambda}^2$  is a Lie algebroid over $\OPDL$ and we have the following commutative diagram:
\begin{equation}\label{eq5}
 \xymatrix{0\ar[r]& I_{\lambda}/I_{\lambda}^2\ar[r]\ar[d]&\Omega(\OPPD)|_{\OPDL}\ar[r]\ar[d]^{\omega}&\Omega(\OPDL)\ar[r]\ar[d]&0,\\
0\ar[r]&T(\OPDL)\ar[r]&T(\OPPD)|_{\OPD}\ar[r]&N_{\OPDL/\OPPD}\ar[r]&0}
\end{equation}
\noindent where $N_{\OPDL/\OPPD}$ denotes the normal bundle of $\OPDL\subset\OPPD$.
\subsection{}\label{sub3} Let $\OPDU$ the sub-scheme of $\OPPD$ consisting of opers that are unramified as  local systems. Clearly we have $\OPDL\subset\OPDU$. Consider now the short exact sequence 
\[
0\rightarrow N_{\OPDL/\OPDU}\rightarrow N_{\OPDL/\OPPD}\rightarrow N_{\OPDU/\OPPD}\rightarrow 0.
\]
\noindent Following \cite{BD2} and \cite{FG3} we have the following.
\begin{prop}\label{p8}
The natural map $\Omega(\OPDL)\rightarrow N_{\OPDL/\OPPD}$ is injective, in particular we have
\[
\Omega(\OPDL)\simeq N_{\OPDL/\OPDU}.
\]
\end{prop}

\subsection{Deformation of the commutative chiral algebra $\ZZ$}
Recall the topological associative algebra $\ZC\simeq Fun(\OPPD)$ attached to $\ZZ$ at the point $x\in X$.\\
 For a dominant weight $\lambda$,  recall that we denote by $\VAC$ the $\gc$-module 
 $$\VAC:=\text{Ind}_{\g[[t]]\oplus \mathbbm{C}}^{\gc}V^{\lambda},$$
 where $V^{\lambda}$ denotes the irreducible finite dimensional representation of $\g$ of highest weight $\lambda$. In particular, recall that the fiber of $\ZZ$ at $x$ is isomorphic to $\zz:=\text{End}(\mathbbm{V}^{0})$.\\
\indent We will now construct a different commutative chiral algebra $\ZL$, isomorphic to $\ZZ$ on $X-x$ and such that $\zl:=\left( \ZL\right)_x\simeq \text{End}(\VAC)$. \\
Note that, in particular, the isomorphism $\ZZ|_{X-x}\simeq \ZL|_{X-x}$ guarantees that the topological associative algebra $\hat{\mathfrak{Z}}_{crit,\lambda}$ attached to $\ZL$ will be isomorphic to $\ZC$.

\subsection{}

Consider the following general set-up.  Let $\A$ be a chiral algebra on $X$, and let $M$ be a cyclic $\A$-module
 supported at $x\in X$, i.e. a module generated by a single element $v\in M$. Then the following is true.

\begin{prop}\label{p3}
Under the above assumptions, there exist a chiral algebra $\A'$ on $X$ such that $\A|_{X-x}\simeq\A'|_{X-x}$ and such that the fiber $i_x^!(\A')[1]$ is isomorphic to $M$.
\end{prop}

\begin{proof}
\quash{
Denote by $\hat{A}_x$ the topological associative algebra attached to it, and consider the natural projection $j_*j^*(\A)\rightarrow \hat{\A}_x.$ Recall that the category of $\A$-modules suppirted at $x$ is equivalent to the category of $\hat{A}_x$-modules. We can therefore regard $M$ as a $j_*j^*(\A)$-module via the above projection and consider the surjection
\[
j_*j^*(\A)\otimes M\rightarrow i_!(M)\rightarrow 0.
\]}
The existence of a cyclic vector $v\in M$ ensures the surjectivity of the map
\[
j_*j^*(\A)\rightarrow i_*(M),
\]
obtained by composing the action map with the map $j_*j^*(\A)\rightarrow j_*j^*(\A)\otimes M
$ given by $a\mapsto a\otimes v$. We define $\A'$ as the kernel of the above map, i.e we have 
\[
0\rightarrow \A'\rightarrow j_*j^*(\A)\rightarrow i_*(M)\rightarrow 0.
\]
It is easy to see that $\A'$ satisfies the required properties.
\end{proof}

By applying the above proposition to the case $\B=\ZZ$, and $M=\B'_x=\zl$, we obtain the desired chiral algebra $\ZL$. 
\quash{
\\Note, in fact, that from the surjection $\ZC\twoheadrightarrow \zl$, it follows that $i_{!}(\zl)$, regarded as a $\ZZ$ module, is cyclic.}

Let now $\OPL$ be the $\DX$-scheme of $\check{\g}$-opers on $X$ that have $\lambda$-regularity at $x\in X$, i.e. that can be written as
\[
\nabla=\nabla_0+\left(\sum_{i}t^{<\alpha_i,\lambda>}\cdot f_i\right)dt+v(t)dt,\,\,\,\,\,\mbox{ for $v(t)\in \check{\mathfrak{b}}(\!(t)\!)$ }
\]
around the point $x\in X$, with coordinate $t$.
We have the following.
\begin{thm}\label{th10}
There exist a canonical isomorphism of $\DX$-algebras
\[
 \ZL\simeq \text{Fun}(\OPL),
\]
in particular we have an isomorhism of commutative algebras $\zl\simeq \text{Fun}(\OPDL)$ and of
 commutative topological algebras $\hat{\mathfrak{Z}}_{crit,\lambda}\simeq\ZC\simeq\text{Fun}(\OPPD).$
\end{thm}

\subsection{Poisson structure on $\ZZ$ and $\ZL$ }
\subsection{} \label{ah}
We will now recall the Poisson structure on $\ZZ$ coming from the chiral algebra $\AC$. \\
Consider the $\mathbbm{C}[[\hbar]]$-family of chiral algebras $\AH$, corresponding to the bilinear form $\kappa=\kappa_{crit}+\hbar \cdot \kappa_{kill}$, and denote by $[\,,\,]_{\AH}$ the Lie$^*$-bracket induced by the chiral bracket on $\AH$. For two sections $z,w\in \ZZ$, consider two arbitrary sections $z_{\hbar},\,w_{\hbar}$ whose values, modulo $\hbar$, are $z$ and $w$ respectively. Then we can define a Poisson bracket $\{z,w\}$ in the following way:
\[
\{z,w\}:=\frac{1}{h}[z_{\hbar},w_{\hbar}]_{\AH}\pmod \hbar.
\] 
The above expression makes sense since $[z_{\hbar},w_{\hbar}]_{\AH}$ vanishes $\pmod \hbar$, and it is easy to see that the element $\{z,w\}$ is indeed in $\Delta_!(\ZZ)$. The above bracket defines a Poisson structure on $\ZZ$.\\
The isomorphism stated in Theorem \ref{th3a} is compatible with this Poisson bracket.\\
\subsection{} Similarly we will now define a Poisson bracket on $\ZL$ that will be compatible with the isomorphism in Theorem \ref{th10}. To do this, we need to introduce a different chiral algebra $\ACL$ with the property that $\AC|_{X-x}\simeq \ACL|_{X-x}$ and with fiber at $x$ isomorphic to $\VAC$. In particular the above conditions will guarantee that the center of $\ACL$ will be the commutative chiral algebra $\ZL$ previously defined.\\
The construction of $\ACL$ follows from Proposition \ref{p3}, by taking $\A$ to be $\AC$, the module $M$ to be $\VAC$ and $v\in \VAC$ any highest weight vector in $V^{\lambda}$.

Now we can proceed as in the previous subsection and define a Poisson structure on $\ZL$. We simply consider the $\mathbbm{C}[[\hbar]]$-family of chiral algebras $\AH'$, given by modifying $\AH$ in the same way as we did for $\AC$, and define the bracket of two elements in $\ZL$ by the same formula. The fact that $\ZL$ is the center of $\ACL$ guarantees that the above expression still makes sense and in fact defines a Lie$^*$-bracket on $\ZL$. Moreover the isomorphism from Theorem \ref{th10} can be upgraded to an isomorphism of Poisson algebras. 
\subsection{}Consider now the chiral algebra $\ZL$ and recall the Lie algebroid $I_{\lambda}/I_{\lambda}^2$ introduced in \ref{alg}. Recall the diagram \ref{eq5}. Because of the Poisson structure on $\ZL$, the $\ZL$-module $\Omega(\ZL)$ acquires a structure of Lie$^*$-algebroid. Denote by $$\omega:\Omega(\ZL)\rightarrow T(\ZL)$$ the anchor map, where $T(\ZL)$ denotes the Lie$^*$-algebroid of vector fields on $\ZL$. \\

\noindent Recall that, for a $\DX$-scheme $\B,$ geometric points of Spec($\B_x$) are the same as horizontal sections of Spec($\B)$ over the formal disc Spec($\mathbbm{C}[[t]]$), for $t$ a coordinate around $x\in X$. Let now $z$ be a geometric point of Spec($\B_x$), corresponding to an horizontal section $\phi_z:\mathbbm{C}[[t]]\rightarrow B$. As it is explained in \cite{FG} Sect. 3.5. we have 
\begin{equation}\label{eq8}
\phi_z^*(T(\B)\!)_x\simeq \left(N_{\B_x/\hat{B}_x}\right) _z,
\end{equation}
where $\hat{\B}_x$ denotes the topological algebra attached to the commutative chiral algebra $\B$ and $\left(N_{\B_x/\hat{B}_x}\right) _z$ denotes the normal bundle at $z$ to Spec($\B_x$) inside Spec($\hat{\B}_x$). \\
In particular, if we take $\B=\ZL$, and take the fiber at $x$ of the map $\omega$, we obtain a map
\[
\omega:\Omega(\ZL)_{x}\simeq \Omega(\zl)\rightarrow N_{\zl/\ZL}.
\]
Under the isomorphism with the space of opers given by Theorem \ref{th10},  and the isomorphism from Proposition \ref{p8}, this map corresponds to the injective map from \ref{sub3}.

\section{Lie$^*$-algebroids and chiral algebroids arising from $\ZL$}\label{3}
 In this section we will assume the reader is familiar with the notion of Lie$^*$-algebroids, chiral extensions of Lie$^*$-algebroids and chiral envelope of such, as introduced in \cite{BD}.\\ Recall the functor $\Gamma^{\lambda}$, regarded as a functor 
 \[
 \Gamma^{\lambda}:\DM^{G[[t]]}\rightarrow \gm_{reg,\tau(\lambda)}.
 \]
 To show the exactness of the above functor, we need to study in more details the category on the left, which involves some constructions from \cite{FG} that we will now recall.
 \subsection{} \label{alg1}
Recall the Lie$^*$-algebroids $\AF$ and the chiral algebroid  $\AREN$ introduced in \cite{FG} Sect. 4.\\
Recall That $\AF$ is a Lie$^*$-algebrois over $\ZZ$ which fits into the exact sequence
\[
0\rightarrow \AC/\ZZ\rightarrow \AF\rightarrow \OO\rightarrow 0,
\]
and it is constructed using the Poisson bracket on the center $\ZZ$ together with the chiral algebras $\AH$ introduced in \ref{ah}. Recall, in fact, that If we denote by $\AH^{\#}$
\[
\AH^{\#}=\left\{  \frac{a}{\hbar}| \, a \in \AH,\, a \pmod \hbar \in \ZZ \right\},
\]
then the Lie$^*$-action of $\AH^{\#}/\ZZ$ on $\ZZ$  via the projection $\AH^{\#}\rightarrow \ZZ$ followed by the poisson bracket on $\ZZ$,  gives the tensor product $\ZZ\otimes \AH^{\#}/\ZZ$ a Lie$^*$-algebroid structure over $\ZZ$. The algebroid $\AF$ is defined in \cite{FG} as a certain quotient of this tensor product.\\
\newcommand{\AFm}{\AF\text{-mod}}

\noindent Recall the category $\AFm$ of $\AF$-modules supported at $x\in X$. It consists of $\M\in \AC\text{-mod}$, with an additional action of the Lie$^*$-algebroid $\AF$ (see \cite{BD} sect. 2.5.16) such that:
\begin{enumerate}
\item As a chiral module over $\ZZ\hookrightarrow \AC$ it is central.\\
\item The two induced Lie$^*$-actions of $\AC/\ZZ$ coincide.\\
\item The chiral action of $\AC$ and the Lie$^*$-action of $\AF$ on $\M$ are compatible with the Lie$^*$-action of $\AF$ on $\AC$. 

\end{enumerate}
\subsection{}Recall now the chiral algebroid $\AREN$. This is a chiral alegbroid (as defined in \cite{BD} sect. 3.9.) fitting into
\[
0\rightarrow \AC\underset{\ZZ}{\otimes}\AC\rightarrow \AREN\rightarrow \OO\rightarrow 0,
\]
constructed using the action on $\ZZ$ given by the Poisson bracket, of  a certain Lie$^*$-algebra mapping to it. \\
\newcommand{\ARENm}{\AREN\text{-mod}^{ch}}

\noindent Recall that the category $\ARENm$ of chiral modules supported at $x$, consists of $\AC\underset{\ZZ}{\otimes}\AC$-modules, with an additional chiral action of $\AREN$ such that the former action coincides with the latter when restricted to $\AC\underset{\ZZ}{\otimes}\AC\subset \AREN$ . Moreover, following \cite{BD} Sect 3.9.24. $\ARENm$ is equivalent to the category of chiral modules over a chiral algebra $U(\AREN)$ called the \emph{chiral envelope} of $\AREN$. \\
\indent According to \cite{FG} Theorem 5.4. the chiral algebra $U(\AREN)$ is isomorphic to the $0$-th part of the sheaf $\DG$ of chiral  differential operators on the loop group $G(\!(t)\!)$ (see the introduction for the definition of the latter).\\
More precisely, as a bimodule over $\gc$, the fiber at any point $x\in X$ of $\DG$ is $G[[t]]$-integrable with respect to both actions, and we have 
two direct sum decompositions of $\left(\mathcal{D}_{crit}\right)_x$ corresponding to the left and right action of $\gc$. These decompositions
coincide up to the involution $\tau$ and we have
\[
\left( \mathcal{D}_{crit}\right)_x=\underset{\lambda \,dominant}{\bigoplus}\left(\mathcal{D}_{crit}\right)_x^{\lambda},
\]
where $\left(\mathcal{D}_{crit}\right)_x^{\lambda}$ is the direct summand supported on the formal completion of 
$ \text{Spec}(\mathfrak{z}_{crit}^{\lambda})$.\\
The $0$-th part of the sheaf $\DG$, is defined as the chiral algebra corresponding to  $\left(\mathcal{D}_{crit}\right)_x^{0}$.
We will denote such a chiral algebra by $\DD$. Hence we have 
\begin{equation}\label{eq6}
\ARENm\simeq \DD\text{mod}.
\end{equation}

\subsection{}\label{ext}Recall now the Lie$^*$-agebroid $\AFT:=\AREN/\ZZ$. An $\AFT$-module $\M$ supported at the point $x$ is a $\AC\otimes\AC$-module 
with an additional action of the Lie$^*$-algebroid $\AFT$ such that:
\begin{enumerate}
\item As a chiral module over $\ZZ\hookrightarrow \AC$ it is central.\\
\item The two induced Lie$^*$-actions of $\AC\underset{\ZZ}{\otimes}\AC$ coincide.\\
\item The chiral action of $\AC\otimes \AC$ and the Lie$^*$-action of $\AFT$ on $\M$ are compatible with the Lie$^*$-action of $\AFT$ on $\AC$. 
\end{enumerate}

\subsection{}\label{algebroids} Using the chiral algebra $\ACL$ given by proposition \ref{p3} and its center $\ZL$ with the Poisson structure explained in \ref{ah}, we can construct a Lie$^*$-algebroid $\AFL$,  a chiral algebroid $\ARENL$, and define the Lie$^*$-algebroid $\AFTL$ to be $\ARENL/\ZL$. The construction is a word by word repetition of what it is done in \cite{FG}, and we will omit it. The first two objects $\AFL$ and $\AFTL$ fit into the following exact sequences
\[
0\rightarrow \ACL/\ZL\rightarrow \AFL\rightarrow \Omega(\ZL)\rightarrow 0,
\]
\[
0\rightarrow \ACL\underset{\ZL}{\otimes}\ACL\rightarrow \ARENL\rightarrow \Omega(\ZL)\rightarrow 0
\]
respectively. As before we define the appropriate category of modules over them in the following way.\\
\newcommand{\AFLm}{\AFL\text{-mod}}
The category $\AFLm$ of modules supported at $x\in X$ consists of $\M\in \ACL\text{-mod}$, with an additional action of the Lie$^*$-algebroid $\AFL$ such that:
\begin{enumerate}
\item As a chiral module over $\ZL\hookrightarrow \ACL$ it is central.\\
\item The two induced Lie$^*$-actions of $\ACL/\ZL$ coincide.\\
\item The chiral action of $\ACL$ and the Lie$^*$-action of $\AFL$ on $\M$ are compatible with the Lie$^*$-action of $\AFL$ on $\ACL$. 
\end{enumerate}
\newcommand{\ARENLm}{\ARENL\text{-mod}^{ch}}

The category $\ARENLm$ of chiral modules supported at $x$ consists of $\ACL\underset{\ZL}{\otimes}\ACL$-modules, with an additional chiral action of $\ARENL$ such that the former action coincides with the latter when restricted to $\ACL\underset{\ZL}{\otimes}\ACL\subset \ARENL$ . 

The category $\AFTL$-mod of modules supported at $x$ consists of a $\M\in \ACL\otimes\ACL$-module 
with an additional action of the Lie$^*$-algebroid $\AFTL$ such that:
\begin{enumerate}
\item As a chiral module over $\ZL\hookrightarrow \ACL$ it is central.\\
\item The two induced Lie$^*$-actions of $\ACL\underset{\ZL}{\otimes}\ACL$ coincide.\\
\item The chiral action of $\ACL\otimes \ACL$ and the Lie$^*$-action of $\AFTL$ on $\M$ are compatible with the Lie$^*$-action of $\AFTL$ on $\ACL$. 
\end{enumerate}

\begin{remark}
Note that a central $\ZZ$-module $\M$ supported at the point $x$, is the same as a $\ZC$-module on which the action of $\ZC$ factors through $\zz$. Similarly, central $\ZL$-modules are the same as $\ZC$-modules on which the action factors through $\zl$. 
\end{remark}

\subsection{The unramified center}
Recall the commutative chiral algebra $\ZZ$, its modification $\ZL$, and the topological associative algebra $\ZC$ attached to them, for every $x\in X$. Recall moreover the commutatie algebras $\zz$ and $\zl$, and the ideals $I_0$ and $I_{\lambda}$ given by the kernel of the projections $\ZC\twoheadrightarrow \zz$ and $\ZZ\twoheadrightarrow \zl$ respectively.\\
\newcommand{\zltt}{\tilde{\mathfrak{z}}_{reg,\lambda}}
Denote by $\zltt$ the formal completion of $\ZC$ with respect to the ideal $I_{\lambda}$, and by Spec$(\ZU$) the subscheme of Spec$(\ZC)$ corresponding to $\OPDU$, i.e. oper that are monodromy free as local systems, under the isomorphism from Theorem \ref{th3a}. We clearly have:
\[
\bigsqcup_{\lambda}\text{Spec}(\zl)\subset \text{Spec}(\ZU)\subset \bigsqcup_{\lambda}\text{Spec}(\zltt).
\]

Consider the functor $\il:\ZC\text{-mod}\rightarrow \zl\text{-mod}$ that takes a $\ZC$-module to its maximal submodule \emph{
scheme-theoretically} supported on Spec$(\zl)$, i.e. for any module $\M$, $\il(\M)$ consists of sections that are 
annihilated by $I_{\lambda}=\text{Ker}(\ZC\rightarrow \zl)$.\\
\newcommand{\iltt}{\tilde{\imath}^{!}_{\lambda}}
Similarly, let us define $\iltt$: $\ZC$-mod$\rightarrow $$\zltt$-mod as the functor that 
 assigns to a $\ZC$-module its maximal 
submodule \emph{set-theoretically} supported on Spec($\zl$) (i.e. supported on Spec($\zltt$)\!). In other words $\iltt(\M)$ admits a filtration whose sub-quotients are annihilated by $I_{\lambda}$.\\

Let $\gm_{reg, \lambda}$ be the subcategory of $\gm$ of modules on which the action of $\ZC$ factors through $\zl$, and let $\gm_{\widetilde{reg},\lambda}$ be the subcategory of $\gm$ consisting of modules that, when regarded as $\ZC$-modules, are set-theoretically supported on $\zl$. 
Denote in the same way the corresponding functors
\[
 \gm \xrightarrow{\iltt} \gm_{\widetilde{reg},\lambda}\subset\zltt\text{mod},\,\,\,\,\,\,\,
\gm^{G[[t]]}\xrightarrow{\iltt} \gm_{\widetilde{reg},\lambda}^{G[[t]]}\subset \zltt\text{mod}.
\]
It is well known that if a $\gc$-module $M$ is $G[[t]]$ integrable, i.e. if $M\in \gm^{G[[t]]}$, then it is supported on the disjoint union of some Spec$(\zltt)$'s. However, something more can be said. In fact we have the following proposition proved in \cite{FG3} Theorem 1.10.
\begin{prop}\label{p4}
The support in Spec($\ZC$) of a $G[[t]]$-integrable module $M$ is contained in Spec$(\ZU)$.
\end{prop}
In other words, the above proposition tells us that the image of the functor 
\[
\gm^{G[[t]]}\xrightarrow{\iltt} \gm_{\widetilde{reg},\lambda}^{G[[t]]},
\]
which is a priori contained in $\zltt$-mod, belongs to $\ZU$-mod. We will denote by $\gm_{unr,\lambda}$ the category consisting of $\gc$-modules, whose support in contained in the $\lambda$ part of $\ZU$.\\

Consider now the functor $\il$ as a functor 
\[
\il:\gm\rightarrow \gm_{reg,\lambda}.
\]
Clearly we have $\il=\il\circ\iltt$. Moreover we have the following:
\begin{prop}\label{p1}
 The functor $\iltt:\gm^{G[[t]]}\rightarrow \gm_{unr,\lambda}^{G[[t]]}$ is exact.
\end{prop}
\begin{proof}
 The proof follows from the following Lemma.
\end{proof}
\begin{lemma}\label{l1}
 For any two positive weights $\lambda$ and $\mu$, we have Spec$(\zl)\cap$ Spec($\mathfrak{z}_{crit,\mu}$)$=\emptyset$ 
\end{lemma}
\begin{proof}
Recall from \cite{FG6} that, for any dominant weight $\lambda$, we have Spec$(\zl)\simeq \OPDL$. Moreover, it is clear from the description of the latter, that for $\lambda$ and $\mu$ dominant weights, we have $\OPDL\cap\OPDM=\emptyset$. 

\end{proof}

\section{Proof of the Main Theorem}\label{4}
Recall now the functor  $\Gamma^{\lambda}:\DM\rightarrow \gm_{reg,\tau(\lambda)}, $
defined as the composition
\begin{equation}\label{eq7}
  \DM\simeq \DGM\xrightarrow{For} (\gc\times \gc)\text{-mod}^{G[[t]]}\xrightarrow{\Hom(\VAC,\,\,)}\gm_{reg,\tau(\lambda)}.
\end{equation}
To show the exactness of the above functor,  we will compose $\Gamma^{\lambda}$ with the forgetful functor to $Vect$, and 
show that the composition
\[
 \Gamma^{\lambda}:D_{crit}(\GR)\text{-mod}\rightarrow Vect
\]
is exact.

\subsection{}
By formula (\ref{eq7}), the functor of global sections $\Gamma^{\lambda}:\DM\rightarrow Vect$ can be viewed as a functor
$\DGM\rightarrow Vect$ given by $\M\rightarrow \Hom_{\gc}(\VAC,\M)$. 
Recall that the module $\VAC$ is supported on Spec$(\zl)$, and moreover, according to \cite{BD2} Sect 8. is a projective generator for the category $\gm_{reg,\lambda}$.\\   The fact that the support of $\VAC$ is Spec$(\zl)$, makes the functor $\Gamma^{\lambda} $ factor as
\[
\begin{array}{ccccccc}
\DGM&\xrightarrow{For}&\gm^{G[[t]]}&\xrightarrow{\il}&\gm_{reg,\lambda}^{G[[t]]}&\longrightarrow&Vect.\\
\M&\longrightarrow&\M&\longrightarrow&\il(\M)&\longrightarrow&\Hom_{\gm_{reg,\lambda}}(\VAC,\il(\M)\!)
\end{array}
\]
\noindent However, because of the projectivity result just recalled,  the last functor is exact. Therefore, 
Theorem \ref{th1} is equivalent to the following:
\begin{thm}\label{th2}
 The composition
\[
 \DGM\rightarrow \gm^{G[[t]]}\xrightarrow{\il} \gm_{reg,\lambda}^{G[[t]]}
\]
is exact.
\end{thm}
\newcommand{\DDm}{\DD\text{-mod}}

\subsection{}\label{red}
Recall the categories $\ARENm\simeq \DD$-mod and $\AFTL$-mod introduced in \ref{alg1}-\ref{algebroids}.\\ Denote by $\DDm_{reg,\lambda}$, $\DDm_{\widetilde{reg},\lambda}$ and $\DDm_{unr,\lambda}$ the pre-image of $\gm_{reg,\lambda}$, $\gm_{\widetilde{reg},\lambda}$ and $\gm_{unr,\lambda}$ under the forgetful functor $$For:\DD\text{-mod}\simeq \ARENm\rightarrow \gc\text{-mod}$$
respectively. 
Recall that a module $\M$ in $\AFTL$-mod is, in particular, a central $\ZL$-modules, and that this is equivalent to the fact that the underlying vector space $M$, viewed as a $\ZC$-module, is supported on Spec$(\zl)$, being $\zl$ the fiber of $\ZL$ at $x$. \\
Note that, moreover,  we have an inclusion
\[
\DDm_{reg,\lambda}\subset \AFTL\text{-mod}.
\]
In fact, a module $\M$ in $\DDm_{reg,\lambda}$ can be seen as a module for $\ARENL$ with a central action of $\ZL$. In other words it receives a Lie$^*$-action of $\ARENL$ that factors through $\ARENL/\ZL=\AFTL$.\\
Consider the functors $\il:\ZC\text{-mod}\rightarrow \zl\text{-mod}$ and $\iltt:\ZC\text{mod}\rightarrow\zltt\text{-mod}$. Denote by the same letters the functors 
\[
\DD\text{-mod}\simeq\ARENm\xrightarrow{\iltt}\ARENm_{\widetilde{reg},\lambda}\,\,\,\text{and}\,\,\, \DD\text{-mod}\xrightarrow{\il} \DDm_{reg,\lambda}\subset \AFTL\text{-mod},
\]
Consider the $G[[t]]$-equivariant categories and the corresponding functors
\[
\DD\text{-mod}^{G[[t]]}\xrightarrow{\iltt}\DD\text{-mod}^{G[[t]]}_{\widetilde{reg},\lambda}\,\,\,\text{and}\,\,\, \DD\text{-mod}^{G[[t]]}\xrightarrow{\il} \AFTL\text{-mod}.
\]
By Proposition \ref{p1}, the first functor is exact. Moreover, by Proposition \ref{p4}, its image is contained in $\DDm_{unr,\lambda}^{G[[t]]}$. Therefore the functor $\il:\DDm^{G[[t]]}\ra\AFTL\text{-mod}^{G[[t]]}$  is the composition of the 
following two functors
\[
\iltt:\DDm^{G[[t]]}\rightarrow \DDm_{unr,\lambda}^{G[[t]]},
\]
\[
(\il)_{unr}: \DDm_{unr,\lambda}^{G[[t]]}\rightarrow \AFTL\text{-mod}^{G[[t]]}.
\]
Here $(\il)_{unr}$ is the restriction of $\il$ to $\DDm_{unr,\lambda}^{G[[t]]}\subset \DDm^{G[[t]]}$.
\subsection{Proof of Theorem \ref{th2}}
We can now proceed to the proof of the Theorem. Recall that we want to show that the composition
\[
 \DGM\rightarrow \gm^{G[[t]]}\xrightarrow{\il} \gm_{reg,\lambda}^{G[[t]]}
\]
is exact. This follows from the following proposition:
\begin{prop}\label{p2}
 The functor \[
 (\il)_{unr}:\DDm_{unr,\lambda}\overset{\sim}{\rightarrow} \AFTL\text{-mod}
\]
\noindent is an equivalence of category.
\end{prop}
\begin{proof}[proof of Theorem \ref{th2}.]
Recall that we have a homeomorphism of chiral algebras
\[
 U(\AREN)\simeq \DD\rightarrow\DG,
\]
\noindent hence the forgetful functor $\DGM\rightarrow \gm$ factors as
\[
 \DGM\rightarrow \DDm\rightarrow \gm.
\]
We have a commutative diagram of functors
\[
 \xymatrix{\DDm^{G[[t]]}\ar[d]^{For}\ar[r]^{\il}&\AFTL\text{-mod}^{G[[t]]}\ar[d]^{For}.\\
\gm^{G[[t]]}\ar[r]^{\il}&\gm_{reg,\lambda}^{G[[t]]}}
\]
Thus to prove Theorem \ref{th2} it is sufficient to show that $\il:\DDm^{G[[t]]}\ra\AFTL\text{-mod}^{G[[t]]}$ is exact.
By the discussion in  \S\ref{red},
the functor $\il$ factors as 
\[
 \DDm^{G[[t]]}\xrightarrow{\iltt}\DDm_{unr,\lambda}^{G[[t]]}\xrightarrow{(\il)_{unr}}\AFT\text{-mod}
.\]
The first arrow is exact by Proposition \ref{p1}, and the second arrow is exact by Proposition \ref{p2}.
This finished the proof.
\end{proof}
\section{Proof of Proposition \ref{p2}}\label{5}
In this section we will prove Proposition \ref{p2}. 
The proof is based on Proposition \ref{kash 1}, a version of Kashiwara lemma for 
chiral modules for a chiral Lie algebroid.
\subsection{}
We will first prove the following proposition, that can be regarded as a different version of Kashiwara's Theorem, where instead of the action of differential operators, we have an action of an algebroid.
\begin{prop}\label{kash 2}
Let $Z$ and  $Y$ be smooth closed subschemes of  a smooth scheme  $X$, such that $Z\subset Y\subset X$. Let $L$ be a Lie algebroid on $X$ preserving $Y$, such that $L\twoheadrightarrow N_{Z/Y}$. Denote by $L\text{-mod}_{Z}^{Y} $ the category of $L$-modules set-theoretically supported on $Z$ and scheme theoretically supported on $Y$. Then this category is generated modules scheme theoretically supported on $Z$. 
\end{prop}

The above proposition basically says that, if we denote by $I$ the sheaf of ideals defining $Z\subset X$, a section $m$ of $\M\in L\text{-mod}_{Z}^{Y} $ can always be written as $\sum_{i} \eta_{i}\cdot m_{i}'$, where $I\cdot m'_{i}=0$, and $\eta_i\in  L$. 
\begin{proof}
Let $I$ be the ideal sheaf of $Z$ in $X$. Choose a basis 
$\{\bar l_i\}_{i=1,...,k}$ of $N_{Z/Y}$ and let $\{l_i\}$ be a lifting to $L$ (It is possible since $L$ maps onto $N_{Z/Y}$).
Choose $\{f_i\}\in I$ such that their image $\{\bar f_i\}$ in $I/I^2$ form a dual of $\{\bar l_i\}$, i.e.,
under the natural paring $$<,>N_{Z/Y}\otimes I/I^2\hookrightarrow N_{Z/X}\otimes I/I^2\ra\mO_Z$$  
we have $<\bar l_i,\bar f_j>=\delta_{i,j}$. 
It implies $l_i(f_j)=\delta_{i,j}\text{\ mod}\ I$, where $l_i(f_j)$ is the natural action of $L$ on $I$.

Let $\mM\in\L\text{-mod}_{Z}^{Y} $. We define $\mM^i= ker(I^i)$. Since $\mM$ is set theoretically supported on
$Z$ we have $\bigcup\mM^i=\mM$. Thus it is enough to show that $\mM^i$ is generated by 
$\mM^1$.  By induction, we can assume $\mM^{i-1}$ is generated by $\mM^1$. 
Let us show that $\mM^i$ is generated by $\mM^1$.
Since $f_i$ maps $\mM^i$ to $\mM^{i-1}$ and $l_i$ maps  $\mM^{i-1}$ to $\mM^{i}$, we have 
a well defined map $\delta=\sum l_if_i:\mM^i/\mM^{i-1}\ra\mM^i/\mM^{i-1}$ 
and 
it is enough to show that $\delta$ is surjective. 

We show that $\delta$ is equal to $(i-1)\text{id}$ on $\mM^i/\mM^{i-1}$.
Let $x\in\mM^i$. We need to show that $$(\delta-(i-1)\!)\cdot x\in\mM^{i-1}$$ and it is equivalent to
show that for any $f\in I$ we have $$f\cdot(\delta-(i-1)\!)\cdot x\in\mM^{i-2}.$$
Since $\mM$ is supported on $Y$, it implies the annihilator  
$N_{Z/Y}^{\bot}\subset I/I^2$ of $N_{Z/Y}$ in $I/I^2$ acts trivially on 
$\mM^i/\mM^{i-1}$. Since $N_{Z/Y}^{\bot}$ and $\{\bar f_i\}$ spans $I/I^2$ it reduce to show
that  $$f_j\cdot(\delta-(i-1)\!)\cdot x\in\mM^{i-2}$$
for $j=1,...,k$. 

Write $f_j\cdot(\delta-(i-1)\!)\cdot x=f_j\cdot(\sum l_if_i-(i-1)\!)\cdot x=\Big\{\sum l_if_i(f_j\cdot x)
-(i-2)(f_j\cdot x)\Big\}+\Big\{l_j(f_j)(f_j\cdot x)-(f_j\cdot x)\Big\}
+\Big\{\sum_{i\neq j} l_i(f_j)(f_j\cdot x)\Big\}.$

In the above expression, the first term is in $\mM^{i-2}$ by induction. The second term is in $\mM^{i-2}$
because $l_j(f_j)-1\in I$ and $f_j\cdot x\in\mM^{i-1}$. The third term is in $\mM^{i-2}$ because
$l_i(f_j)\in I$ for $i\neq j$ and $f_j\cdot x\in\mM^{i-1}$. The proof is finished.
\end{proof}

\subsection{}
In this section
we generalize Proposition \ref{kash 2} to chiral Lie algebroids.
We begin with the review of the result of \cite[\S 3.9]{BD}.
Let $(R,B,\mL,\mL^{\flat})$ be a chiral Lie algebroids (cf. \cite[\S 3.9.6]{BD}).
We assume that  $R$, $B$ are $\mO_X$-flat and $\mL$ is $R$-flat.
Following \cite[\S 3.9.24]{BD} 
we denote by  $\mL^{\flat}\text{-mod}^{ch}$  the category of chiral 
$(B,\mL^{\flat})$-module and 
$\mL^{\flat}\text{-mod}$ the category of $(B,\mL^{\flat})$-module. We have 
a natural adjoint pair:
$$Ind:\mL^{\flat}\text{-mod}\rightleftarrows\mL^{\flat}\text{-mod}^{ch}:For$$
where $For$ is the forgetful functor and $Ind$ is its left adjoint. 

For a $\mM\in\mL\text{-mod}$ we define the PBW filtration on $Ind(\mM)$
as the image by the chiral action $j_*j^*U(B,\mL^{\flat})\boxtimes\mM$.
If $\mM$ is a central $R$-module, then
we have a natural isomorphism 
\begin{equation}\label{pbw}
\mM\otimes_{R}\text{Sym}_{R}\mL\is\text{gr}Ind(\mM).
\end{equation}

We denote by $\sX=\text{Spec}(\hat R_x)$ and $\sY=\text{Spec}(R_x)$. Let $\hat\sY$ be 
the formal neighborhood of $\sY$in $\sX$, and let $\sY'$ be an ind-subscheme of $\hat\sY$
containing $\sY$ stable under the action of $\mL$.
Let $\mL^{\flat}\text{-mod}^{ch}_{\sY'}$ be the category of chiral $\mL^{\flat}$-module
supported on $\sY'$. Let $\mL^{\flat}\text{-mod}_{cent}\subset\mL^{\flat}\text{-mod}$ be the subcategory
$(B,\mL^{\flat})$-modules which are central as $R$-module.

Assume now that $\mL$ is elliptic, i.e., the arch map $\omega:\mL\ra\Theta(R)$ is injective and 
the the co-kernel of  $\omega$ is projective $R$-module of finite rank.
We further assume 
the image of  $\omega_x:\mL_x\hookrightarrow\Theta(R)_x\is N_{\sY/\sX}$ is 
equal to $N_{\sY/\sY'}$.

\begin{lemma}\label{filt}
For any $\mM\in\mL^{\flat}\text{-mod}_{cent}$, we have 
\item
a) $Ind(\mM)\in\mL^{\flat}\text{-mod}^{ch}_{\sY'}$.
\\
b) $F^i(Ind(\mM)\!)=Ker(I^i:Ind(\mM)\ra Ind(\mM)\!)$, where $I$ is the ideal sheaf of $\sY$ in $\sX$.
\end{lemma}
\begin{proof}
Let us first prove $a)$. We prove it by induction on the PBW filtration on $Ind(\mM)$. 
For simplicity, we denote by $F^i:=F^iInd(\mM)$, the $i$-the filtration.
By definition, $F^1$ is the image of the canonical map $\mM\ra Ind(\mM)$,
therefore, is supported on $\sY\subset\sY'$. Now by induction, we assume that
$F^{i-1}Ind(\mM)$ is supported on $\sY'$. Since $F^{i-1}Ind(\mM)$ is stable 
under $B$ and $(j_*j^*(\mL^{\flat})\boxtimes F^{i-1}Ind(\mM)\!)$ maps surjectively onto
$F^iInd(\mM)$, taking into account that $\mL^{\flat}/B\is\mL$ and $\sY'$ is stable 
under the action of $\mL$, we obtain $F^iInd(\mM)$ is supported on $\sY'$.

Proof of $b)$. Since $gr Ind(\mM)$ is central $R$-module, it implies $F^i\subset Ker(I^i)$.
To see the other inclusion, considering 
$$f^i:I/I^2\otimes_{\sY} (F^{i+1}/F^i) \ra (F^i/F^{i-1}).$$
We claim that for any nonzero element $\bar x\in F^{i+1}/F^i$ the image of $\bar x$
under $f^i$ is nonzero. 
The claim will imply the inclusion. Indeed, for any $x\neq 0\in Ker(I^i)$, let 
$k$ be the smallest non-negative integer such that $x\in F^{i+k}$. If $k=0$, we are done. If not, 
we have  $\bar x\neq 0\in F^{i+k}/F^{i+k-1}$.
Considering the map 
$$f:(I/I^2)^{\otimes i}\otimes F^{i+k}/F^{i+k-1}\ra F^k/F^{k-1}.$$
Since $x\in Ker(I^i)$, the image of $\bar x$ under $f$ is zero. 
On the other hand, the claim implies the image is nonzero. Contradiction.

Proof of the claim. It is enough to show that the dual of $f^i$
$$(f^i)^\vee:F^{i+1}/F^i\ra F^i/F^{i-1}\otimes_{\sY}N_{\sY/\sX}$$
is injective. But by \ref{pbw}, above map
is obtained by tensoring with $\mM$ from
$$\text{Sym}_{\sY}^i(\mL_x)\ra\text{Sym}_{\sY}^{i-1}(\mL_x)\otimes\mL_x
\ra\text{Sym}_{\sY}^{i-1}(\mL_x)\otimes N_{\sY/\sX}$$
which is injective since $\mL$ is elliptic.
\end{proof}

By above Lemma, we obtain a functor 
 $i_!:\mL^{\flat}\text{-mod}_{cent}\ra\mL^{\flat}\text{-mod}^{ch}_{\sY'}$
by restriction of $Ind$ to $\mL^{\flat}\text{-mod}_{cent}$.
Let $i^!:\mL^{\flat}\text{-mod}^{ch}_{\sY'}\ra\mL^{\flat}\text{-mod}_{cent}$
be the functor of taking maximal submodule supported on $\sY$.
By construction and Lemma \ref{filt}, it is easy to see that $i^!$ is the right adjoint of $i_!$
and the adjunction map $\mM\ra i^!\circ i_!(\mM)$ is an isomorphism. 
\begin{prop}\label{kash 1}
The functor $i_!:\mL^{\flat}\text{-mod}_{cent}\ra\mL^{\flat}\text{-mod}^{ch}_{\sY'}$ is an equivalence of categories 
with inverse given by $i^!$.
\end{prop}
\subsection{Proof of Proposition \ref{kash 1}}
We start with the following easy Lemma:
\begin{lemma}
Let $\mathrm C_1$ $\mathrm C_2$ be two abelian categoies and 
let $$F:\mathrm C_1\rightleftarrows\mathrm C_2:G$$
be a pair of adjoint functors. We assume that $G$ is conservative.
If for any $\mM\in\mathrm C_1$,  $\mM\ra G\circ F(\mM)$ is an isomorphism 
and for any $\mN\in\mathrm C_2$,
 $F\circ G(\mN)\ra\mN$ is surjective, then $F$ is an equivalence of 
 categories with inverse $G$.
\end{lemma}

We apply above Lemma to our setting, i.e. $\mathrm C_1=\mL^{\flat}\text{-mod}_{cent}$,
 $\mathrm C_2=\mL^{\flat}\text{-mod}^{ch}_{\sY'}$, $F=i_!$ and $G=i^!$. 
We already showed that $\mM\ra i_!\circ i^!(\mM)$ is an isomorphism.
Since $\sY'\subset\hat\sY$, it implies $i^!$ is conservative.  Therefore, it remains to
prove the following:
\begin{lemma}
For any $\mN\in\mL^{\flat}\text{-mod}_{\sY'}^{ch}$, the adjunction map 
$i_!\circ i^!(\mN)\ra\mN$ is surjective.
\end{lemma}
\begin{proof}
We applied Proposition \ref{kash 2} to $\sY\subset\sY'\subset\sX$ and 
the algebroid $L=H^0_{dr}(D_x^*,\mL^{\flat})$. By assumption, the morphism 
$L\ra H^0_{dr}(D_x^*,\mL)\ra \mL_x\is N_{\sY/\sY'}$
is surjective, thus Proposition implies $L\boxtimes i^!(\mN)$ maps surjectively onto $\mN$. 
Since the image of $i_!\circ i^!(\mN)\ra\mN$ contains the image of $L\boxtimes i^!(\mN)$, it 
implies $i_!\circ i^!(\mN)\ra\mN$ is surjective.
 \end{proof}
\subsection{Proof of Proposition \ref{p2}}
Let us apply above discussion to the case $\mL^{\flat}=\ARENL$, $\sX=\OPPD$, $\sY=\OPDL$ and 
$\sY'=\OPDU$.  By definition, we have an equivalence of categories 
$$\ARENL\text{-mod}_{cent}\is\AFTL\text{-mod}.$$
On the other hand, since $\AREN$ and $\ARENL$ are isomorphic on $X-x$, there is an 
equivalence of categories $\ARENL\text{-mod}^{ch}\is\AREN\text{-mod}^{ch}$.
Moreover, above equivalence 
induces an equivalence of sub-categories 
$$\ARENL\text{-mod}^{ch}_{\sY'}\is\AREN\text{-mod}^{ch}_{unr,\lambda}$$  
and we have the following commutative diagram:
\[\xymatrix{\ARENL\text{-mod}^{ch}_{\sY'}\ar[r]\ar[d]^{i^!}&\AREN\text{-mod}^{ch}_{unr,\lambda}\ar[d]^{(i_{\lambda}^!)_{unr}}
\\\ARENL\text{-mod}_{cent}\ar[r]&\AFTL\text{-mod}}\]
Thus, Proposition \ref{kash 1} implies the following:
\begin{thm}[Proposition \ref{p2}]
The functor \[(i_{\lambda}^!)_{unr}:\DDm_{unr,\lambda}\is\AREN\text{-mod}^{ch}_{unr,\lambda}\ra\AFTL\text{-mod}\]
is an equivalence of categories. 
\end{thm}

\section{Faithfulness}\label{6}
In this section we show that the functor $\Gamma^{\lambda}:\DM\ra\gc\text{-mod}$
is faithful.
The proof is similar to the case $\lambda=0$ using ideal of 
Harish-Chandra action of \cite[\S 7.14]{BD2}. Namely, let $K\subset G[[t]]$ be a
compact open subgroup. Let $\DM^K$ be the category of $K$-equvariant $D_{crit}$-module
on $\GR$ and let $\gc\text{-mod}^K$ be the category of $K$-integrable $\gc$-modules.
It is clear that $\Gamma^{\lambda}$ maps $\DM^K$ to $\gc\text{-mod}^K$.

Let $D^b(D_{crit}(\GRK)\!)$ be the bounded derived category of $D_{crit}$-modules
on $\GRK$. We have the convolution functor
\[\star:D^b(D_{crit}(\GRK)\!)\times\DM^K\ra D^b(D_{crit}(\GR)\!).\]
Moreover, we have the action functor
\[\star:D^b(D_{crit}(\GRK)\!)\times\gc\text{-mod}^K\ra D^b(\gc\text{-mod}).\]

\begin{lemma}\label{hc}
The functor $\Gamma^\lambda=R\Gamma^{\lambda}:D^b(\DM^K)\ra D^b(\gc\text{-mod}^K)$
intertwines the $D^b(D_{crit}(\GRK)\!)$-action.
\end{lemma}

\begin{proof}
\quash{
Using the map $\pi:G(\!(t)\!)\ra\GR$, 
giving any $\mM\in\DM$ we can produce a chiral $\DG$-module supported at $x$
by considering \[\Gamma(G(\!(t)\!),\pi^*(\mM)\!).\]}
Let $\mF\in\DM$ and let $M_{\mF}\in\DGM$ be the corresponding 
chiral $\DG$-module under Theorem \ref{t1}.
By the result in \cite[Section 1.12]{FG3},  we have 
\[M_{\mF}=\mF\star\mathcal{D}_{crit,x}\] where $\mathcal{D}_{crit,x}$ is the fiber of $\DG$
at $x$ which is a $(\gc,G[[t]])$-bimodule.
Thus for any $\mF_1\in D^b(D_{crit}(\GRK)\!)$ and $\mF_2\in\DM^K$, we have 
$\Gamma^{\lambda}(\GR,\mF_1\star\mF_2)=\text{Hom}_{\gc}(\mathbb V^{\lambda},M_{(\mF_1\star \mF_2)})=
\text{Hom}_{\gc}(\mathbb V^{\lambda},\mF_1\star \mF_2\star\mathcal{D}_{crit,x})=
\mF_1\star \text{Hom}_{\gc}(\mathbb V^{\lambda},\mF_2\star\mathcal{D}_{crit,x})
=\mF_1\star \Gamma^{\lambda}(\GR,\mF_2).$

\end{proof}

\begin{prop}
The functor $\Gamma^{\lambda}$ is faithfull
, i.e., for any $\mM\neq 0\in\DM$ we have $\Gamma^{\lambda}(\GR,\mM)\neq 0$.
\end{prop}
\begin{proof}
We follow \cite{FG} Section 9.10.
It is shown in \cite[Lemma 9.11]{FG}, for any $\mM\neq 0\in\DM^K$ there exists
$\mM'\in D_{crit}\text{-mod}(\GRK)^{G[[t]]}$ such that $\mM'\star\mM\neq 0\in D^b(D_{crit}(\GR)^{G[[t]]})$.
Therefore, by Lemma \ref{hc}, it is enough to show that $\Gamma^{\lambda}(\mM)\neq 0$
for $\mM\neq 0\in\DM^{G^{[[t]]}}$.  
Recall that the convolution functor $\star$ on $\DM^{G[[t]]}$ is exact and geometric Satake identify 
the monoidal category $(\DM^{G[[t]]},\star)$ with the category of representation of the dual group. 
In particular, it implies for any non-zero $\mM\in\DM^{G[[t]]}$ we can find  $\mM'\in\DM^{G[[t]]}$
such that  there is an injection $\delta_e\hookrightarrow\mM'\star\mM$, where $\delta_e$ is the delta $D$-module 
at the unit $e\in\GR$. Since $\Gamma^{\lambda}$ is exact, we obtain an injection
\[0\neq\mathbb V^{\tau(\lambda)}=\Gamma^{\lambda}(\delta_e)\hookrightarrow \Gamma^{\lambda}(\mM'\star\mM)=\mM'\star\Gamma^{\lambda}(\mM)\]
and it implies $\Gamma^{\lambda}(\mM)\neq 0$.
\end{proof}

\vspace{2.0cm}


\newpage

\begin{thebibliography}{99}
\bibitem[AG]{AG} S. Arkhipov and D. Gaitsgory,{\it Differential operators on the loop group via chiral algebras}, IMRN 2002,
no. 4, 165–210.
\bibitem[BD]{BD} A. Beilinson, V. Drinfeld, {\it Chiral algebras}. American Mathematical Society Colloquium Publications, 51.
\bibitem[BD2]{BD2} A. Beilinson and V. Drinfeld, {\it Quantization of Hitchin integrable system and Hecke eigensheaves},
Preprint, available at \url{http://www.math.harvard.edu/~gaitsgde/grad_2009/}.
\bibitem[FF]{FF}  B. Feigin and E. Frenkel, {\it Affine Kac-Moody algebras at the critical level and Gelfand-Dikii algebras},
Infinite Analysis, eds. A. Tsuchiya, T. Eguchi, M. Jimbo, Adv. Ser. in Math. Phys. 16, 197–215,
Singapore: World Scientific, 1992.
\bibitem[F1]{F1} E. Frenkel, {\it W-algebras and Langlands-Drinfeld correspondence}, Proceedings of
Cargese Summer School, 1991.
\bibitem[F2]{F2} E. Frenkel, {\it Lectures on Wakimoto modules, opers and the center at the critical level}. 
arXiv:math/0210029.
\bibitem[FG]{FG} E. Frenkel, D. Gaitsgory, {\it D-modules on the affine Grassmannian and representation of affine Kac-Moody algebras.}  Duke Math. J. 125 (2004), no. 2, 279Ð327.
\bibitem[FG2]{FG2} E. Frenkel, D. Gaitsgory, {\it Local Geometric Langlands correspondence: the Spherical Case}.
Algebraic analysis and around, 167Ð186, Adv. Stud. Pure Math., 54.
\bibitem[FG3]{FG3} E. Frenkel, D. Gaitsgory, \emph{Fusion and Convolution: applications to affine Kac-Moody algebras at the critical level}, Pure and Applied Math Quart. 2, (2006) 1255-1312.
\bibitem[FG4]{FG4} E. Frenkel, D. Gaitsgory, {\it Local Geometric Langlands correspondence and affine Kac-Moody algebras}. Algebraic geometry and number theory, 69Ð260, Progr. Math., 253.
\bibitem[FG5]{FG5}E. Frenkel, D. Gaitsgory, {\it Localization of $\gc$-modules on the affine Grassmannian}.
Ann. of Math. (2) 170 (2009), no. 3, 1339Ð1381.
\bibitem[FG6]{FG6}E. Frenkel, D. Gaitsgory, {\it Weyl modules and Opers without monodromy}. arXiv:0706.3725.
\bibitem[FKW]{FK} E. Frnkel, V. Kac, M. Wakimoto, {\it Character and Fusion Rules for W-algebras via Quantized Drinfeld-Sokolov Reduction}, Commun. Math. Phys. 147, 295-328 (1992).
\bibitem[K]{K} M. Kashiwara, {\it D-modules and Microlocal Calculus}, Translations of Mathematical Monographs, Vol. 217. 



\end{thebibliography}
\end{document}